\documentclass[11pt]{amsart} \usepackage{latexsym, amssymb, stmaryrd}
\usepackage[T1]{fontenc}

\DeclareTextSymbol{\thh}{T1}{254}

\newtheorem{thm}{Theorem}[section]
\newtheorem{lemma}[thm]{Lemma}
\newtheorem{prop}[thm]{Proposition}
\newtheorem{cor}[thm]{Corollary}

\theoremstyle{definition}

\newtheorem{rmk}[thm]{Remark}
\newtheorem{rmks}[thm]{Remarks}

\newtheorem{fact}[thm]{Fact}

\newcommand{\R}{\mathbb{R}}
\newcommand{\Z}{\mathbb{Z}}

\newcommand{\curly}[1]{\mathcal{#1}}

\newcommand{\la}{\curly{L}}

\makeatletter

\def\indsym#1#2{%
  \setbox0=\hbox{$\m@th#1x$}%
  \kern\wd0%
  \hbox to 0pt{\hss$\m@th#1\mid$\hbox to 0pt{$\m@th#1^{#2}$}\hss}%
  \lower.9\ht0\hbox to 0pt{\hss$\m@th#1\smile$\hss}%
  \kern\wd0}

\def\nindsym#1#2{%
  \setbox0=\hbox{$\m@th#1x$}%
  \kern\wd0%
  \hbox to 0pt{\hss$\m@th#1\not$\kern1.4\wd0\hss}
  \hbox to 0pt{\hss$\m@th#1\mid$\hbox to 0pt{$\m@th#1^{\,#2}$}\hss}%
  \lower.9\ht0\hbox to 0pt{\hss$\m@th#1\smile$\hss}%
  \kern\wd0}

\def\dotminussym#1#2{%
  \setbox0=\hbox{$\m@th#1-$}%
  \kern.5\wd0%
  \hbox to 0pt{\hss\hbox{$\m@th#1-$}\hss}%
  \raise.6\ht0\hbox to 0pt{\hss$\m@th#1.$\hss}%
  \kern.5\wd0}
\newcommand{\dotminus}{\mathbin{\mathpalette\dotminussym{}}}

\def \r { {\mathbb R} }
\def \<{\langle}
\def \>{\rangle}
\def \n {\mathbb N}

\def \*Z {{{^*}\Z}}

\def \((  {(\!(}
\def \)) {)\!)}

\def \int{\operatorname{int}}
\numberwithin{equation}{section}

\def \u{\mathcal{U}}

\def\R{\mathcal R}
\def \Th{\operatorname{Th}}

\def \tr{\operatorname{tr}}

\allowdisplaybreaks[2]

\begin{document}

\title[Computability and CEP]{A computability-theoretic reformulation of the Connes Embedding Problem}
\author{Isaac Goldbring and Bradd Hart}
\thanks{Goldbring's work was partially supported by NSF grant DMS-1262210.}


\address {Department of Mathematics, Statistics, and Computer Science, University of Illinois at Chicago, Science and Engineering Offices M/C 249, 851 S. Morgan St., Chicago, IL, 60607-7045}
\email{isaac@math.uic.edu}
\urladdr{http://www.math.uic.edu/~isaac}

\address{Department of Mathematics and Statistics, McMaster University, 1280 Main Street W., Hamilton, Ontario, Canada L8S 4K1}
\email{hartb@mcmaster.ca}
\urladdr{http://www.math.mcmaster.ca/~bradd}

\begin{abstract}
The Connes Embedding Problem (CEP) asks whether every separable II$_1$ factor embeds into an ultrapower of the hyperfinite II$_1$ factor.  We show that the CEP is equivalent to the computability of the universal theory of every type II$_1$ von Neumann algebra.  We also derive some further computability-theoretic consequences of the CEP.
\end{abstract}

\maketitle

\section{Introduction}

Let $\R$ denote the hyperfinite II$_1$ factor.  In his seminal paper \cite{Co}, Connes proved that $L(\mathbb F_n)$, the group von Neumann algebra of the free group of rank $n$, embeds into an ultrapower $\R^\u$ of $\R$.  He then casually remarked that ``Apparently such an imbedding ought to exist for all (separable) II$_1$ factors...''  This seemingly innocuous statement is now referred to as the \emph{Connes Embedding Problem} (hereafter referred to as the CEP) and is arguably the most important open problem in the theory of II$_1$ factors.  Due to the work of Kirchberg and others, there are now many equivalent formulations of the CEP spanning nearly all parts of operator algebras as well as various areas outside of operator algebras such as geometric group theory and noncommutative real algebraic geometry; see \cite{Cap} for a survey on the many equivalents of CEP.

In the article \cite{FHS2}, it is shown how to view tracial von Neumann algebras as structures in a particular \emph{continuous} logic suited for studying structures from analysis.  Moreover, is shown how the classes of tracial von Neumann algebras and II$_1$ factors form axiomatizable classes in this logic.  In the sequel \cite{FHS3}, the authors observe that CEP is actually equivalent to the logical statement that every II$_1$ factor has the same \emph{universal theory} as $\R$.  Roughly speaking, this means that, for any quantifer-free formula $\varphi(x_1,\ldots,x_n)$ (which is essentially just a continuous function applied to the traces of various ${}^*$polynomials in the variables $x_1,\ldots,x_n$), where $x_1,\ldots,x_n$ range over the unit ball of a II$_1$ factor, we have
$$\sup\{\varphi(\vec a) \ : \ \vec a\in \R_1\}=\sup\{\varphi(\vec b) \ : \ \vec b\in M_1\}$$ for any II$_1$ factor $M$.  
It is then immediate that this latter statement is equivalent to its \emph{existential} version, obtained by replacing $\sup$'s by $\inf$'s, which is often called the \emph{Microstate Conjecture}, a well-known equivalent to the CEP.

In this paper, we use the CEP to show that the universal theory of $\R$ is \emph{computable}, meaning that there is an algorithm such that, given any quantifier-free formula $\varphi(\vec x)$ and any dyadic rational $\epsilon>0$ as input, returns an interval $I\subseteq \r$ (with dyadic rational endpoints) of length at most $\epsilon$ such that $\sup\{\varphi(\vec a) \ : \ a\in \R_1\}\in I$.  (Of course, such an algorithm will then also exist for existential sentences.)

Trivially, the result of the previous paragraph shows that, assuming CEP, every type II$_1$ von Neumann algebra has a computable universal theory (as its universal theory coincides with the universal theory of $\R$).  What if instead one started with the assumption that every type II$_1$ von Neumann algebra has a computable universal theory?  Consider the algorithm that starts comparing the values of universal sentences in $\R$ with other II$_1$ factors; if CEP failed, this algorithm would eventually tell us so, otherwise, the algorithm would run forever.  Thus, the assumption that every type II$_1$ von Neumann algebra has a computable universal theory would only help one verify that CEP failed (if CEP were in fact false).  Unfortunately, this line of thought is doomed to fail.  Indeed, we prove that, if every type II$_1$ von Neumann algebra has a computable universal theory, then the CEP holds.  The key point here is to show that if the CEP fails, then it ``fails very badly'' in the sense that if there are at least two distinct universal theories of type II$_1$ algebras, then there are continuum many such universal theories. 

In the last section, we derive further computability-theoretic consequences of the CEP.

A curious byproduct of our results is that both sides of the CEP vs. NCEP (not CEP) debate will find something useful here to consider in their endeavors.  Indeed, if one is trying to prove that CEP is true, then it seems a priori easier to show that all universal theories of type II$_1$ algebras are decidable rather than equal.  And if one is trying to prove NCEP, then the strong computability-theoretic consequences derived from CEP should be seen as strong evidence that the CEP is far too strong to be true.

We would like to thank David Sherman for a helpful conversation regarding this project.

\section{Prerequisites from Logic}

In this paper, by a tracial von Neumann algebra we mean a pair $(A,\tr)$, where $A$ is a von Neumann algebra and $\tr$ is a faithful, normal trace on $A$.  However, we often suppress mention of the trace and simply say ``Let $A$ be a tracial von Neumann algebra...''  (This causes no confusion when $A$ is a II$_1$ factor for then the trace on $A$ is unique.)

For a von Neumann algebra $A$, we let $A_1$ denote the operator norm unit ball.

Let $\mathcal F$ denote the set of all $*$-polynomials $p(x_1,\ldots,x_n)$ ($n\geq 0$) such that, for \emph{any} von Neumann algebra $A$, we have $p(A_1^n)\subseteq A_1$.  For example, the following functions belong to $\mathcal F$:
\begin{itemize}
\item the ``constant symbols'' $0$ and $1$ (thought of as $0$-ary functions);
\item $x\mapsto x^*$;
\item $x\mapsto \lambda x$ ($|\lambda|\leq 1$);
\item $(x,y)\mapsto xy$;
\item $(x,y)\mapsto \frac{x+y}{2}$.
\end{itemize}

We then work in the \emph{language} $\la:=\mathcal F\cup \{\tr_\R,\tr_\Im,d\}$, where $\tr_\Re$ (resp. $\tr_\Im$) denote the real (resp. imaginary) parts of the trace and $d$ denotes the metric on $A_1$ given by $d(x,y):=\|x-y\|_2$.  We can then formulate certain properties of tracial von Neumann algebras using the language $\la$ as follows.

\emph{Basic $\la$-formulae} will be formulae of the form $\tr_\Re(p(\vec x))$ or $\tr_\Im(p(\vec x))$ for $p\in \mathcal F$.  \emph{Quantifier-free $\la$-formulae} are formulae of the form $f(\varphi_1(\vec x),\ldots,\varphi_m(\vec x))$, where $f:\r^m\to \r$ is a continuous function and $\varphi_1,\ldots,\varphi_m$ are basic $\la$-formulae.  Finally, an arbitrary $\la$-formula is of the form 
$$Q^1_{x_1\in B_1}\cdots Q^k_{x_k\in B_1}\varphi(x_1,\ldots,x_n),$$ where $k\leq n$, $\varphi(x_1,\ldots,x_n)$ is a quantifier-free formula, and each $Q^i$ is either $\sup$ or $\inf$; we think of these $Q_i$'s as \emph{quantifiers over the unit ball of the algebra}.  If each $Q^i$ is $\sup$ (resp. $\inf$), then we say that the formula is \emph{universal} (resp. \emph{existential}).

\begin{rmks}

\

\begin{enumerate}
\item Our setup here is a bit more special than the general treatment of continuous logic in \cite{BBHU}, but a dense set of the formulae in \cite{BBHU} are logically equivalent to formulae in the above form, so there is no loss of generality in our treatment here.
\item In order to keep the set of formulae ``separable'', when forming the set of quantifier-free formulae, we restrict ourselves to a countable dense subset of the set of all continuous functions $\r^m\to \r$ as $m$ ranges over $\n$.  In fact, one can take this countable dense set to be ``finitely generated'' which is important for our computability-theoretic considerations.  (See \cite{BP}.)
\end{enumerate}
\end{rmks}

Suppose that $\varphi(\vec x)$ is a formula, $A$ is a tracial von Neumann algebra, and $\vec a\in A_1^n$, where $n$ is the length of the tuple $\vec x$.  We let $\varphi(\vec a)^A$ denote the real number obtained by replacing the variables $\vec x$ with the tuple $\vec a$; we may think of $\varphi(\vec a)^A$ as the truth value of $\varphi(\vec x)$ in $A$ when $\vec x$ is replaced by $\vec a$.  For example, if $\varphi(x_1)$ is the formula $\sup_{x_2}d(x_1x_2,x_2x_1)$, then $\varphi(a)^A=0$ if and only if $a$ is in the center of $A$.

If $\varphi(\vec x)$ is a formula, then there is a bounded interval $[m_\varphi,M_\varphi]\subseteq \r$ such that, for any tracial von Neumann algebra $A$ and any $\vec a\in A$, we have $\varphi(\vec a)^A\in [m_\varphi,M_\varphi]$.  

If $\varphi$ has no \emph{free variables} (that is, all variables occurring in $\varphi$ are bound by some quantifier), then we say that $\varphi$ is a \emph{sentence} and we observe that $\varphi(\vec a)^A$ is the same as $\vec a$ ranges over all tuples of $A$ of the appropriate length, whence we denote it simply by $\varphi^A$.  Given a tracial von Neumann algebra, the \emph{theory of $A$} is the function $\Th(A)$ which maps the sentence $\varphi$ to the real number $\varphi^A$.  (Sometimes authors define $\Th(A)$ to consists of the set of sentences $\varphi$ for which $\varphi^A=0$; since $\Th(A)$, as we have defined it, is determined by its zeroset, these two formulations are equivalent.)  If we restrict the function $\Th(A)$ to the set of all universal (resp. existential) sentences, the resulting function is defined to be the \emph{universal} (resp. \emph{existential}) theory of $A$, denoted $\Th_\forall(A)$ (resp. $\Th_\exists(A)$).  We should also mention that, as a consequence of \emph{\L os' theorem}, we have $\Th(A)=\Th(A^\u)$ for any ultrafilter $\u$.  

\begin{rmk}[For the logicians]\label{logicians}
In what follows, we will restrict ourselves to $\la$-structures that are tracial von Neumann algebras.  We can do this because it is shown that the class of (unit balls of) tracial von Neumann algebras forms a universally axiomatizable class of $\la$-structures.
\end{rmk}

Let $T$ be a set of $\la$-sentences.  We say that a tracial von Neumann algebra $A$ \emph{models} $T$, written $A\models T$, if $\varphi^A=0$ for each $\varphi\in T$.  It is shown in \cite{FHS2} that there is a set $T_{II_1}$ of $\la$-sentences such that $A\models T_{II_1}$ if and only if $A$ is a II$_1$ factor.  In fact, there is a \emph{recursive} such set $T_{II_1}$, meaning that there is an algorithm which determines, upon input a sentence $\sigma$, whether or not $\sigma$ belongs to $T_{II_1}$.  The aforementioned observation will be crucial for what is to follow and so we isolate it:

\begin{fact}\label{recaxiom}
The class of II$_1$ factors is recursively axiomatizable.
\end{fact}

Up until now, we have been treating tracial von Neumann algebras \emph{semantically}.  It will be crucial to also treat them \emph{syntactically}.  In \cite{BP}, a \emph{proof system} for continuous logic is established.  In our context, this gives meaning to the phrase ``the axioms $T_{II_1}$ can prove the sentence $\sigma$,'' which we denote $T_{II_1}\vdash \sigma$.

\begin{fact}
The set $\{\sigma \ : \ T_{II_1}\vdash \sigma\}$ is \emph{recursively enumerable}, meaning that there is an algorithm that runs forever and continually returns those $\sigma$ for which $T_{II_1}\vdash \sigma$.
\end{fact}

\begin{proof}
This follows immediately from the existence of the proof system developed in \cite{BP} together with Fact \ref{recaxiom}.
\end{proof}

There is a connection between the semantic and syntactic treatments developed above (which \cite{BP} refers to as ``Pavelka-style completeness'').  Let $\dotminus:\r^2\to \r$ be the function $x\dotminus y:=\max(x-y,0)$ and let $\mathbb D$ denote the set of dyadic rational numbers.

\begin{fact}(\cite[Corollary 9.8]{BP})
For a sentence $\varphi$, we have 
$$\sup\{\varphi^A \ : \ A\models T_{II_1}\}=\inf\{r\in \mathbb D^{>0} \ : \ T_{II_1}\vdash \varphi\dotminus r\}.$$  We denote this common value by $\varphi_{T_{II_1}}$.
\end{fact}

\begin{rmk}
By Downward L\"owenheim-Skolem, every tracial von Neumann algebra has a separable subalgebra with the same theory.  Consequently, we have that
$$\varphi_{T_{II_1}}=\sup\{\varphi^A \ : \ A\models T_{II_1} \text{ and }A \text{ is separable}\}.$$
\end{rmk}

\subsection{CEP and Model Theory}

At this point, it is convenient to recall the connection between CEP and model theory.  If $A,B$ are tracial von Neumann algebras and $A$ is a subalgebra of $B$, then $\Th_\forall(A)\leq \Th_\forall(B)$ (as functions).  Since $\R$ embeds into any II$_1$ factor, we have that $\Th_\forall(\R)\leq \Th_\forall(A)$ for every II$_1$ factor $A$.  If $A$ is a $\R^\omega$-embeddable II$_1$ factor, then certainly $\Th_\forall(A)\leq \Th_\forall(\R)$ (as $\Th(\R)=\Th(\R^\u)$).  Conversely, suppose that $A$ is a separable tracial von Neumann algebra such that $\Th_\forall(A)\leq \Th_\forall(\R)$.  It is then a standard fact of model theory that $A$ is $\R^\omega$-embeddable.  We thus see that CEP is equivalent to the statement that, for every II$_1$ factor $A$, we have that $\Th_\forall(A)=\Th_\forall(\R)$.  (Actually, we just saw that CEP is equivalent to the statement that, for every separable tracial von Neumann algebra $A$ containing $\R$, we have $\Th_\forall(A)=\Th_\forall(\R)$.)  As a side remark, it is an exercise to see that, for tracial von Neumann algebras $A$ and $B$, we have $\Th_\forall(A)=\Th_\forall(B)$ if and only if $\Th_\exists(A)=\Th_\exists(B)$, which is easily seen to be equivalent to the Microstate Conjecture.

\section{CEP implies Computability}\label{CEPimplies}

\emph{In this section, we assume that CEP holds.}  For ease of notation, we set $T:=T_{II_1}$.

\begin{lemma}
Suppose that $\sigma$ is universal.  Then $\sigma_T=\sigma^\R$.
\end{lemma}

\begin{proof}
By definition, $\sigma^\R\leq \sigma_T$.  Now fix a separable II$_1$ factor $M$; we must show $\sigma^M\leq \sigma^\R$.  This follows immediately from the fact that $M$ is $\R^\omega$ embeddable.
\end{proof}

\begin{lemma}
Suppose that $\sigma$ is existential.  Then $\sigma_T=\sigma^\R$.
\end{lemma}

\begin{proof}
Again, it suffices to show that $\sigma^M\leq \sigma^\R$ for arbitrary $M\models T$.  But this follows from the fact that $M$ contains a copy of $\R$.
\end{proof}

\begin{cor}
If $\sigma$ is a universal sentence, then $(M_\sigma\dotminus \sigma)_T=M_\sigma\dotminus \sigma_T$.
\end{cor}

\begin{proof}
Observe that $M_\sigma\dotminus \sigma$ is logically equivalent to an existential sentence.  Using the previous two lemmas, we have
$$(M_\sigma\dotminus \sigma)_T=(M_\sigma\dotminus \sigma)^\R=M_\sigma\dotminus \sigma^\R=M_\sigma\dotminus \sigma_T.$$
\end{proof}

If $A$ is a tracial von Neumann algebra, we say that $\Th_\forall(A)$ is \emph{computable} if there is an algorithm such that, upon inputs universal sentence $\sigma$ and positive dyadic rational number $\epsilon$, returns an interval $I\subseteq \r$ of length at most $\epsilon$ with dyadic rational endpoints such that $\sigma^A\in I$.  One defines $\Th_\exists(A)$ being computable in an analogous way.

\begin{rmk}
This is \textbf{not} the same notion of computable theory as defined in \cite{BP} but is more appropriate for our needs.
\end{rmk}
%

\begin{cor}
$\Th_\forall(\R)$ and $\Th_\exists(\R)$ are computable.
\end{cor}

\begin{proof}
Here is the algorithm:  given universal $\sigma$ and positive dyadic rational $\epsilon$, run all proofs from $T$ and wait until you see that $T\vdash \sigma \dotminus r$ and $T\vdash (M_\sigma\dotminus \sigma)\dotminus s$ where $r-(M_\sigma-s)\leq \epsilon$.  By the previous corollary, this algorithm will eventually halt and the interval $[M_\sigma-s,r]$ will be the desired interval.
\end{proof}
%
%
%

\section{Computability implies CEP}

Recall that NCEP implies that there are at least two distinct universal (equivalently existential) theories of type II$_1$ algebras.  In fact:

\begin{prop}
Suppose that CEP fails.  Then there are continuum many different universal (equivalently existential) theories of type II$_1$ algebras.  In fact, there is a single existential sentence $\sigma$ such that $\sigma^M$ takes on continuum many values as $M$ ranges over all type II$_1$ algebras.
\end{prop}

\begin{proof}
For $N\in \n$, $A$ a type II$_1$ algebra, $a$ a tuple from $M$, and $\epsilon>0$, let $\sigma_{N,A,a,\epsilon}$ be the existential sentence
$$\inf_x \max_{\deg p\leq N} \max(|\tr_\Re(p(x))-\tr_\Re(p(a))|,|\tr_\Im(p(x))-\tr_\Im(p(a))|).$$
Since CEP fails, there are $N$, $A$, $a$, and $\epsilon>0$ such that $\sigma_{N,A,a,\epsilon}^\R>0$.  (Of course $\sigma_{N,A,a,\epsilon}^A=0$.)  For simplicity, set $\sigma:=\sigma_{N,A,a,\epsilon}$ and $r:=\sigma^\R$.  For each $t\in [0,1]$, set $A_t:=t\R\oplus (1-t)A$, which denotes the direct sum of $\R$ and $A$ with trace $\tr_t:=t\tr_\R+(1-t)\tr_A$.  Note that each $A_t$ is a type II$_1$ algebra and the map $t\mapsto \sigma^{A_t}:[0,1]\to \r$ is continuous.  Since $\sigma^{A_0}=0$ and $\sigma^{A_1}=r$, the proof of the proposition is complete.
\end{proof}

\begin{cor}
Suppose that the universal theory of every type II$_1$ algebra is computable.  Then CEP holds.
\end{cor}

\begin{proof}
Suppose that CEP fails.  By the previous lemma, there are uncountably many universal theories of type II$_1$ algebras.  But there are only countably many programs that could be computing universal theories of type II$_1$ algebras, whence not every type II$_1$ algebra has a computable universal theory.
\end{proof}

\section{Further computability-theoretic consequences of the CEP}

\emph{In this section, we assume that CEP holds} and we derive some further computability-theoretic results.  Unlike Section \ref{CEPimplies}, in this section, we let $T$ denote the set of sentences whose models are the tracial von Neumann algebras (see Remark \ref{logicians}).

Fix a separable II$_1$ factor $A$ with enumerated subset $X=(a_0,a_1,a_2,\ldots)$ that generates $A$ (as a von Neumann algebra).  We now pass to a language $\la_X$ containing $\la$ obtained by adding to $\la$ new constant symbols for each $a_i$.  We now add to $T$ sentences of the form $\max(r_n\dotminus f(\vec a),f(\vec a)\dotminus s_n)$, where $f\in \mathcal F$ and $(r_n,s_n)$ is a sequence of intervals of dyadic rationals containing $f(\vec a)$ with $s_n-r_n\to 0$; we call the resulting theory $T_{(A,X)}$.  (In model-theoretic lingo:  we are just adding the \emph{atomic diagram} of $A$ to $T$.)  Note that a model of $T_{(A,X)}$ is a tracial von Neumann algebra $B$ whose interpretations of the new constants generate a von Neumann subalgebra of $B$ isomorphic to $A$.  

We say that $(A,X)$ as above is \emph{recursively presented} if there is an algorithm that enumerates each sequence of intervals $(r_n,s_n)$ for each $f\in \mathcal F$.  It is a standard construction in recursion theory to code a recursively presented tracial von Neumann algebra $(A,X)$ by a single natural number, which we refer to as the \emph{G\"odel code} of $(A,X)$.  

Fix a recursively presented II$_1$ factor $(A,X)$.  Suppose that $\sigma=\sup_x \varphi(x)$ is a universal sentence and $\epsilon$ is a positive dyadic rational.  Then clearly there is $n\in \n$ such that $\sigma^A\leq \max_{i\leq n}\varphi(a_i)^A+\epsilon$; we will say that such an $n$ is \emph{good for $(A,X,\sigma,\epsilon)$}.  Consider the following algorithmic question:  is there a way of computably determining some $n$ that is good for $(A,X,\sigma,\epsilon)$?  The next result tells us that CEP implies that there is a \emph{single} algorithm that works for all recursively presented $(A,X)$ and all $\sigma$ and $\epsilon$.

\begin{thm}
There is a computable partial function $f:\n\times \n\times \mathbb D^{>0}\rightharpoonup \n$ such that, if $e$ is the G\"odel code of a recursively presented separable II$_1$ factor $(A,X)$ and $n$ is the G\"odel code of a universal sentence $\sigma=\sup_x\varphi(x)$, then $f(m,n,\epsilon)$ is good for $(A,X,\sigma,\epsilon)$.
\end{thm}

\begin{proof}
Here is the algorithm for determining $f(m,n,\epsilon)$.  First, use the computability of $\Th_\forall(\R)$ to determine an interval $I=[c,d]\subseteq \r$ with $|I|\leq \frac{\epsilon}{2}$ such that $\sigma^\R\in I$.  By CEP, $\sigma^\R=\sigma^A$.  We claim that there is an $N$ such that $c-\frac{\epsilon}{2}\leq \varphi(a_N)^A$.  Indeed, there is $N$ such that $\sigma^A-\frac{\epsilon}{2}\leq \varphi(a_N)^A$.   For such an $N$, we have that $c-\frac{\epsilon}{2}\leq \varphi(a_N)^A\leq \sigma^A\leq d$ and $d-(c-\frac{\epsilon}{2})\leq \epsilon$, whence $N$ is good for $(A,X,\sigma,\epsilon)$.  Now we just start computing $\varphi(a_i)^A$ (which we can do since $(A,X)$ is recursively presented) and wait until we reach $N$ with $c-\frac{\epsilon}{2}\leq \varphi(a_N)^A$.
\end{proof}

Note that there is a countable $X\subseteq \R$ such that $(\R,X)$ is recursively presented.  In the rest of this paper, we fix such an $X$ and let $T_\R:=T_{(\R,X)}$ and let $\R_X$ denote the obvious expansion of $\R$ to an $\la_X$-structure.

In the next proof, we will need the following fact (see \cite[Lemma 3.1]{GHS}):

\begin{fact}
For any nonprincipal ultrafilter $\u$ on $\n$, any embedding $h:\R\to \R^\u$ is \emph{elementary}, that is, for any formula $\varphi(\vec x)$, and any tuple $\vec a\in \R$, we have $\varphi^\R(\vec a)=\varphi^{\R^\u}(h(\vec a))$.
\end{fact}

\begin{lemma}
Suppose that $\sigma$ is a universal or existential $\la_X$-sentence.  Then $\sigma_{T_\R}=\sigma^{\R_X}$.
\end{lemma}

\begin{proof}
As in Section \ref{CEPimplies}, we need only show that $\sigma^M\leq \sigma^{\R_X}$ for every $M\models T_\R$.  First suppose that $\sigma$ is existential, say $\sigma=\inf_x \varphi(c_a,x)$, where $a$ is a tuple from $X$ and $c_a$ is the corresponding tuple of constants.  Let $i:\R\to M$ be the embedding of $\R$ into $M$ determined by setting $i(a):=c_a^M$ for every $a\in X$.  Then
$$\sigma^M=\inf\{\varphi(i(a),b)^M \ : \ b\in M\}\leq \inf\{\varphi(i(a),i(d))^M \ : \ d\in \R\}=\sigma^{\R_X}.$$

Now suppose that $\sigma$ is universal, say $\sigma=\sup_x \varphi(c_a,x)$.  Fix an embedding $j:M\to \R^\u$.  Then
$$\sigma^M=\sup\{\varphi(i(a),b)^M \ : \ b\in M\}\leq \sup\{\varphi(ji(a),d)^{\R^\u} \ : \ d\in \R^\u\}=\sigma^{\R_X},$$ since $ji:\R\to \R^\u$ is elementary. 
\end{proof}


\begin{cor}
$\Th_\forall(\R_X)$ and $\Th_\exists(\R_X)$ are computable.
\end{cor}

\begin{proof}
This follows from the previous lemma just as in Section \ref{CEPimplies}.
\end{proof}

Define $\Th_{\exists\forall}(\R)$ to be the restriction of $\Th(\R)$ to the set of formulae of the form
$$Q^1_{x_1\in B_1}\cdots Q^k_{x_k\in B_k}\varphi(x_1,\ldots,x_n),$$ where $\varphi$ is quantifier-free, $k\leq n$, and such that there is $l\in \{1,\ldots,k\}$ such that $Q^i=\inf$ for $i\in \{1,\ldots,l\}$ and $Q^i=\sup$ for $i\in \{l+1,k\}$.

We say that $\Th_{\exists\forall}(\R)$ is \emph{upper computably enumerable} if there is an algorithm that enumerates all sentences of the form $\sigma\dotminus s$, where $\sigma$ is an $\exists\forall$-sentence and $s$ is a dyadic rational with $\sigma^\R<s$.

\begin{cor}
$\Th_{\exists\forall}(\R)$ is upper computably enumerable.
\end{cor}

\begin{proof}
Consider (for simplicity) the sentence $\inf_x\sup_y\varphi(x,y)$.  For each $a\in X$ and $\epsilon \in \mathbb D^{>0}$, use the previous corollary to find an interval $I=[r,s]$ with dyadic endpoints of length $\leq \epsilon$ such that $\sup_y\varphi(a,y)^\R\in I$.  We then add the condition $\inf_x\sup_y \varphi(a,y)\leq s$ to our enumeration.  We claim that this algorithm shows that $\Th_{\exists \forall}(\R)$ is upper computably enumerable.  Indeed, suppose that $\inf_x\sup_y\varphi(x,y)=s$.  Fix $s'\in \mathbb D$, $s<s'$.  Fix $\delta \in \mathbb D^{>0}$ such that $s+2\delta<s'$.  We claim that when the algorithm encounters $a\in X$ such that $\sup_y\varphi(a,y)^\R\leq s+\delta$, our algorithm will let us know that $\inf_x\sup_y\varphi(x,y)\leq s'$.  Indeed, our algorithm will tell us that $\inf_x\sup_y \varphi(x,y)\leq d$, where $d\in \mathbb D^{>0}$ and $d\leq \sup_y\varphi(a,y)^\R+\delta$.
\end{proof}

\end{document}